\documentclass{amsart}
\usepackage{amsfonts}
\usepackage{amsmath}
\usepackage{amssymb}
\usepackage{graphicx}%
\newtheorem{theorem}{Theorem}
\theoremstyle{plain}

\newtheorem{lemma}{Lemma}

\newtheorem{proposition}{Proposition}
\newtheorem{remark}{Remark}

\numberwithin{equation}{section}
\def\d{\partial}

\def\1bar{\overline{1}}

\def\Im{\operatorname{Im}}
\def\Re{\operatorname{Re}}

\begin{document}
\title{An Obata-type Theorem in CR Geometry}
\author{Song-Ying Li}
\address{Department of Mathematics, University of California, Irvine, CA
92697; School of Mathematics and Computer Science, Fujian Normal University, Fuzhou, China}
\email{sli@math.uci.edu}
\author{Xiaodong Wang}
\address{Department of Mathematics, Michigan State University, East Lansing,
MI 48824}
\email{xwang@math.msu.edu}
\thanks{The second author was partially supported by NSF grant DMS-0905904.}
\begin{abstract}
We discuss a sharp lower bound for the first positive eigenvalue of the sublaplacian on a closed, strictly pseudoconvex pseudohermitian manifold of dimension $2m+1\geq 5$. We prove that the equality holds iff the manifold is equivalent to the CR sphere up to a scaling. For this purpose, we establish an Obata-type theorem in CR geometry which characterizes the CR sphere in terms of a nonzero function satisfying a certain overdetermined system. Similar results are proved in dimension 3 under an additional condition.

\end{abstract}
\maketitle

\section{Introduction\protect\bigskip}

In Riemannian geometry, estimates on the first positive eigenvalue of the
Laplace operator have played important roles and there have been many beautiful results.
We refer the reader to the books Chavel \cite{C} and Schoen-Yau \cite{SY}. The following theorem is a classic result.

\begin{theorem}
(Lichnerowicz-Obata) Let $\left( M^{n},g\right) $ be a closed Riemannian
manifold with $Ric\geq \left( n-1\right) \kappa $, where $\kappa $ is a
positive constant. Then the first positive eigenvalue of Laplacian satisfies%
\begin{equation*}
\lambda _{1}\geq n\kappa. \leqno(1.1)
\end{equation*}%
Moreover, equality holds iff $M$ is isometric to a round sphere.
\end{theorem}

The estimate $\lambda _{1}\geq n\kappa $ was proved by Lichnerowicz \cite{L} in 1958.
The characterization of the equality case was established by Obata \cite{O} in 1962. In fact, he deduced it from the 
following more general
\begin{theorem}\label{OT}
(Obata \cite{O}) Suppose $\left( N^{n},g\right) $ is a complete Riemannian manifold
and $u$ a smooth, nonzero function on $N$ satisfying $D^{2}u=-c^{2}ug$, then 
$N$ is isometric to a sphere $\mathbb{S}^{n}\left( c\right) $ of radius $1/c$
in the Euclidean space $\mathbb{R}^{n+1}$.
\end{theorem}

In CR geometry,   we have the most basic example of a second order differential
operator which is subelliptic, namely the sublaplacian $\Delta _{b}$.
On a closed pseudo-hermitian manifold, the sublaplacian $\Delta
_{b}$ still defines a selfadjoint operator with a discrete spectrum 
\begin{equation*}
\lambda _{0}=0<\lambda _{1}\leq \lambda _{2}\leq \cdots \leqno(1.2)
\end{equation*}%
with $\lim_{k\rightarrow \infty }\lambda _{k}=+\infty $. One would naturally hope that
the study of these eigenvalues in CR geometry will be as fruitful as in Riemannian geometry. An analogue of the Lichnerowicz estimate for the
sublaplacian on a strictly pseudoconvex pseudo-Hermitian manifold
 $(M^{2m+1}, \theta)$  was  proved by Greenleaf in \cite{G}  for $m\ge 3$ and by Li and Luk \cite{LL} for $m=2$. 
Later it was pointed out  that there was an error in the proof of the Bochner formula  in \cite{G}. Due to this error, the Bochner formula as well as the 
CR-Lichnerowicz theorem in \cite{G, LL} are not correctly formulated. 
The corrected statement is

\begin{theorem}
\label{Greenleaf}Let $(M,\theta)$ be a closed, strictly pseudoconvex pseudohermitian manifold of dimension 
$2m+1\geq 5$. Suppose that  the Webster pseudo Ricci curvature and the pseudo torsion satisfy %
\begin{equation*}
Ric\left( X,X\right) -\frac{m+1}{2}Tor\left( X,X\right) \geq \kappa
\left\vert X\right\vert ^{2} \leqno(1.3)
\end{equation*}
for all $X\in T^{1,0}\left( M\right) $, where $\kappa $ is a positive constant. Then the first positive eigenvalue of $%
-\Delta _{b}$ satisfies 
\begin{equation*}
\lambda _{1}\geq \frac{m}{m+1}\kappa. \leqno(1.4)
\end{equation*}
\end{theorem}

The estimate is sharp as one can verify that equality holds on the CR sphere
$$\mathbb{S}^{2m+1}=\{z\in \mathbb{C}^{m+1}:|z|=1\}$$ 
with the standard pseudohermitian structure
$$\theta_0=2\sqrt{-1}\overline{\partial}(|z|^2-1).$$
The natural question whether the equality case characterizes the CR sphere was
not addressed in \cite{G}. The torsion appearing in (1.3) is a major new obstacle 
comparing with the Riemanian case. This question has been recently studied by
several authors and partial results have been established. Chang and Chiu \cite{CC1}
proved that the equality case characterizes the CR\ sphere if $M$ has zero
torsion. In \cite{CW}, Chang and Wu proved the rigidity under the condition that the torsion satisfies certain identities involving its covariant derivatives.  
Ivanov and Vassilev \cite{IV} proved the same conclusion under the
condition that the divergence of the torsion is zero. 
Li and Tran \cite{LT} considered the special case that $M$ is a real ellipsoid $E(A)$ in $\mathbb{C}^{m+1}$.
They computed $\kappa$ explicitly and proved that the equality, $\lambda_1=m\kappa/(m+1)$  implies $E(A)$ is the sphere. 
 But in general it is very difficult to handle the torsion.  
 
 In this paper, we provide a new method which can handle the torsion
 and yields an affirmative answer to this question in the general case.

\begin{theorem}
\label{Main}If equality holds in Theorem \ref{Greenleaf}, then $(M,\theta)$ is equivalent to the
sphere $\mathbb{S}^{2m+1}$ with the standard pseudohermitian structure $\theta_0$ up to a scaling, i.e. there exists a CR diffeomorphism $F:M\rightarrow\mathbb{S}^{2m+1}$ such that
$F^*\theta_0=c\theta$ for some constant $c>0$.
\end{theorem}

In fact our proof yields the following more general result  which can be viewed as the CR analogue of Theorem \ref{OT} (for notation see Section 2). 

\begin{theorem}\label{crot}
Let $M$ be a closed pseudohermitian manifold of dimension $2m+1\geq 5$.
Suppose there exists  a real nonzero function $u\in C^{\infty }\left(
M\right) $ satisfying 
\begin{eqnarray*}
u_{\alpha ,\beta } &=&0, \\
u_{\alpha ,\overline{\beta }} &=&\left( -\frac{\kappa}{2(m+1)}u+\frac{\sqrt{-1}}{2}%
u_{0}\right) h_{\alpha\overline{\beta }},
\end{eqnarray*}%
for some constant $\kappa>0$.
Then $M$ is equivalent to the sphere $\mathbb{S}^{2m+1}$ with the standard pseudohermitian structure up to a
scaling.
\end{theorem}

The 3-dimensional case is more subtle. It is not clear if these results are true in 3 dimensions. Partial results with additional conditions are discussed in the last section. 

The approach in \cite{CC1} is to consider a family of adapted Riemannian metrics and try to apply the Lichnerowiz-Obata theorem. This 
approach requires very complicated calculations to relate the various CR quantities and the corresponding Riemannian ones. In December 2010, the authors found 
a new approach working directly with the Riemannian Hessian of the eigenfunction. With this approach we generalized the Chang-Chiu result to show that rigidity holds provided the 
double divergence of the torsion vanishes (see Remark \ref{aQ} in Section4). 
In their preprint \cite{IV},  Ivanov and Vassilev found the same strategy independently and proved rigidity under the condition that the divergence of the torsion is zero.
But to solve the general case requires a new ingredient. We employ a delicate integration by parts argument which requires a good understanding of the
critical set of the eigenfunction. 

The paper is organized as follows. In Section 2, we review some basic facts
in CR geometry. In Section 3, following  the argument of Greenleaf we present the proof of Theorem 2
 with all the necessary corrections. Theorem 3 is proved in Sections 4 and 5.  Finally, in Section 6,
we discuss the $3$-dimensional case.

{\bf Acknowledgement.} The authors wish to thank the referees for carefully reading the paper and making valuable suggestions.

\section{Preliminaries}

Let $\left( M,\theta ,J\right) $ be a strictly pseudoconvex pseudohermitian
manifold of dimension $2m+1$. Thus $G_{\theta }=d\theta \left( \cdot ,J\cdot
\right) $ defines a Riemannian metric on the contact distribution $H\left(
M\right) =\ker \theta $. As usual, we set $T^{1,0}\left( M\right) =\{w-\sqrt{-1}J w : w\in H\left( M\right)
\}\subset T\left( M\right) \otimes 
\mathbb{C}
$ and $T^{0,1}\left( M\right) =\overline{T^{1,0}\left( M\right) }$.
Let $T$ be the Reeb vector field and extend $J$ to
an endomorphism $\phi $ on $TM$ by defining $\phi \left( T\right) =0$. We
have a natural Riemannian metric $g_{\theta}$ on $M$ such that $TM=\mathbb{R%
}T\oplus H\left( M\right) $ is an orthogonal decomposition and $g_{\theta
}\left( T,T\right) =1$. In the following, we will simply denote $g_{\theta }$
by $\left\langle \cdot ,\cdot \right\rangle $. Let $\widetilde{\nabla }$ be
the Levi-Civita connection of $g_{\theta }$ while $\nabla $ is the
Tanaka-Webster connection. For basic facts on CR geometry, one can consult the recent book \cite{DT} or the 
original papers by Tanaka \cite{T} and Webster \cite{W}.

Recall that the Tanaka-Webster connection is compatible with the metric  $g_{\theta }$, but it has a non-trivial torsion.
 The torsion $\tau $ satisfies%
\begin{eqnarray*}
\tau \left( Z,W\right) &=&0, \\
\tau \left( Z,\overline{W}\right) &=&\omega \left( Z,\overline{W}\right) T,
\\
\tau \left( T,J\cdot \right) &=&-J\tau \left( T,\cdot \right)
\end{eqnarray*}%
for any $Z,W\in T^{1,0}\left( M\right) $, where $\omega=d\theta$.
We define $A:T\left( M\right) \rightarrow T\left( M\right) $ by $AX=\tau
\left( T,X\right) $. It is customary to simply call $A$ the torsion of the CR manifold.  It is easy to see that $A$ is symmetric.
Moreover $AT=0,AH\left( M\right) \subset
H\left( M\right) $ and $A\phi X=-\phi AX$.

The following formula gives the difference between the two connections $%
\widetilde{\nabla }$ and $\nabla $. The proof is based on straightforward
calculation and can be found in \cite{DT}.

\thinspace

\begin{proposition}
\label{compare}We have%
\begin{eqnarray*}
\widetilde{\nabla }_{X}Y &=&\nabla _{X}Y+\theta \left( Y\right) AX+\frac{1}{2%
}\left( \theta \left( Y\right) \phi X+\theta \left( X\right) \phi Y\right) \\
&&-\left[ \left\langle AX,Y\right\rangle +\frac{1}{2}\omega \left(
X,Y\right) \right] T.
\end{eqnarray*}
\end{proposition}

\begin{remark}\label{one}
\bigskip We have 
\begin{equation*}
\widetilde{\nabla }_{X}T=AX+\frac{1}{2}\phi \left( X\right) .
\end{equation*}%
In particular,  $\widetilde{\nabla }_{T}T=0$. If $X$ and $Y$ are both
horizontal i.e. $X,Y\in H(M)$, then 
\begin{equation*}
\widetilde{\nabla }_{X}Y=\nabla _{X}Y-\left[ \left\langle AX,Y\right\rangle +%
\frac{1}{2}\omega \left( X,Y\right) \right] T.
\end{equation*}
\end{remark}

In the following, we will always work with a local frame $\{T_{\alpha
}:\alpha =1,\cdots ,m\}$ for $T^{1,0}\left( M\right) $. Then $\{T_{\alpha
},T_{\overline{\alpha }}=\overline{T_{\alpha }},T_{0}=T\}$ is a local frame
for $T\left( M\right) \otimes 
\mathbb{C}
$. Let $h_{\alpha \overline{\beta }}=-i\omega \left( T_{\alpha },T_{%
\overline{\beta }}\right) =g_{\theta }\left( T_{\alpha },T_{\overline{\beta }%
}\right) $ be the components of the Levi form. For a smooth function $u$ on $%
M$, we will use notations such as $u_{\alpha ,\overline{\beta }}$ to denoted
its covariant derivatives with respect to  the Tanaka-Webster connection $\nabla $.
Let $D^{2}u$ be the Hessian of $u$ with respect to the Riemannian metric $%
g_{\theta }$.

By straightforward calculation using Proposition \ref{compare}, one can
derive

\begin{proposition}
\label{hessian}We have the following formulas%
\begin{eqnarray*}
&&D^{2}u\left( T,T\right) =u_{0,0}, \\
&&D^{2}u\left( T,T_{\alpha }\right) =u_{\alpha ,0}-\frac{\sqrt{-1}}{2}%
u_{\alpha }, \\
&&D^{2}u\left( T_{\alpha },T_{\beta }\right) =u_{\alpha,\beta }+A_{\alpha \beta
}u_{0}, \\
&&D^{2}u\left( T_{\alpha },T_{\overline{\beta }}\right) =u_{\alpha ,%
\overline{\beta }}-\frac{\sqrt{-1}}{2}h_{\alpha \overline{\beta }}u_{0}.
\end{eqnarray*}
\end{proposition}

In doing calculations we will need to use repeatedly the following formulas
which can be found in \cite{DT} or \cite{Lee}.

\begin{proposition}
\label{rule}We have the following formulas%
\begin{eqnarray*}
u_{0,\alpha } &=&u_{\alpha,0}+A_{\alpha }^{\overline{\beta }}u_{\overline{\beta }%
}, \\
u_{\alpha,\beta }&=&u_{\beta ,\alpha }, \\
u_{\alpha ,\overline{\beta }} &=&u_{\overline{\beta },\alpha }+\sqrt{-1}%
h_{\alpha \overline{\beta }}u_{0}, \\
u_{\alpha ,0\beta } &=&u_{\alpha ,\beta 0}+A_{\beta }^{\overline{\gamma }%
}u_{\alpha ,\overline{\gamma }}+R_{\beta 0\alpha }^{\sigma }u_{\sigma } \\
&=&u_{\alpha ,\beta 0}+A_{\beta }^{\overline{\gamma }}u_{\alpha ,\overline{%
\gamma }}-A_{\alpha \beta ,\overline{\gamma }}h^{\sigma \overline{\gamma }%
}u_{\sigma }, \\
u_{\alpha ,0\overline{\beta }} &=&u_{\alpha ,\overline{\beta }0}+u_{a,\gamma
}h^{\gamma \overline{\nu }}A_{\overline{\nu }\overline{\beta }}+h^{\gamma 
\overline{\nu }}A_{\overline{\nu }\overline{\beta },\alpha }u_{\gamma }, \\
u_{\alpha ,\overline{\beta }\overline{\gamma }} &=&u_{\alpha ,\overline{%
\gamma }\overline{\beta }}+\sqrt{-1}\left( h_{\alpha \overline{\beta }}A_{%
\overline{\gamma }}^{\sigma }-h_{\alpha \overline{\gamma }}A_{\overline{%
\beta }}^{\sigma }\right) u_{\sigma }, \\
u_{\alpha ,\beta \overline{\gamma }} &=&u_{\alpha ,\overline{\gamma }\beta }+%
\sqrt{-1}h_{\beta \overline{\gamma }}u_{\alpha ,0}-R_{\beta \overline{\gamma 
}\alpha }^{\sigma }u_{\sigma }, \\
u_{\alpha ,\beta \gamma } &=&u_{\alpha ,\gamma \beta }-R_{\beta \gamma
\alpha }^{\sigma }u_{\sigma } \\
&=&u_{\alpha ,\gamma \beta }+\sqrt{-1}\left( A_{\alpha \gamma }u_{\beta
}-A_{\alpha \beta }u_{\gamma }\right) .
\end{eqnarray*}
\end{proposition}

\begin{remark}
Our convention for the curvature tensor is 
\begin{equation*}
R\left( X,Y,Z,W\right) =\left\langle -\nabla _{X}\nabla _{Y}Z+\nabla
_{Y}\nabla _{X}Z+\nabla _{\left[ X,Y\right] }Z,W\right\rangle .
\end{equation*}
\end{remark}

\section{The estimate on $\protect\lambda _{1}$}

In this section, we prove the estimate on $\lambda _{1}$ following
Greenleaf \cite{G}. This serves two purposes. First, there is a mistake in 
\cite{G} as pointed out by \cite{GL} and \cite{CC1}. This has caused some confusion (e.g.
see the presentation in \cite{DT}) and we hope to clarify the whole
situation. Secondly, we need to analyze the proof when we address the
equality case.

From now on, we always work with a local unitary frame $\{T_{\alpha }:\alpha
=1,\cdots ,m\}$ for $T^{1,0}\left( M\right) $. Given a smooth function $u$,
its sublaplacian is given by%
\begin{equation*}
\Delta _{b}u=\sum_{\alpha }u_{\alpha ,\overline{\alpha }}+u_{\overline{%
\alpha },\alpha }.
\end{equation*}%
We have the following Bochner formula.

\begin{theorem}\label{Boc}
Let $\left\vert \partial_b u\right\vert ^{2}=u_{\alpha }u_{\overline{\alpha }}$%
. Then 
\begin{eqnarray*}
\frac{1}{2}\Delta _{b}\left\vert \partial_b u\right\vert ^{2} &=&\left\vert
u_{\alpha ,\beta }\right\vert ^{2}+\left\vert u_{\alpha ,\overline{\beta }%
}\right\vert ^{2}+\frac{1}{2}\left[ \left( \Delta _{b}u\right) _{\alpha }u_{%
\overline{\alpha }}+\left( \Delta _{b}u\right) _{\overline{\alpha }%
}u_{\alpha }\right] \\
&&+R_{\alpha \overline{\sigma }}u_{\sigma }u_{\overline{\alpha }}+\frac{m}{2}%
\sqrt{-1}\left[ A_{\alpha \sigma }u_{\overline{\alpha }}u_{\overline{\sigma }%
}-A_{\overline{\alpha }\overline{\sigma }}u_{\sigma }u_{\alpha }\right] \\
&&+\sqrt{-1}\left( u_{\overline{\beta }}u_{\beta ,0}-u_{\beta }u_{\overline{%
\beta },0}\right) .
\end{eqnarray*}
\end{theorem}

\begin{remark}
This was first derived by Greenleaf \cite{G}. But due to a mistake in
calculation pointed out by  \cite{GL} and \cite{CC1}, the coefficient $\frac{m}{2}$ on the RHS was mistaken to be $\frac{m-2}{2}$.
\end{remark}

The following formulas are also derived in \cite{G}.
\begin{lemma}
\label{intfor}Let $u$ be a smooth function on a closed pseudohermitian manifold $M$ of dimension $2m+1$. We have the following integral equalities%
\begin{equation*}
\int_M \sqrt{-1}\left( u_{\overline{\beta }}u_{\beta ,0}-u_{\beta }u_{%
\overline{\beta },0}\right) =\frac{2}{m}\int_M \left\vert u_{\alpha ,\overline{%
\beta }}\right\vert ^{2}-\left\vert u_{\alpha ,\beta }\right\vert
^{2}-R_{\alpha \overline{\sigma }}u_{\sigma }u_{\overline{\alpha }},
\end{equation*}%
\begin{equation*}
\int_M \sqrt{-1}\left( u_{\overline{\beta }}u_{\beta ,0}-u_{\beta }u_{%
\overline{\beta },0}\right) =\int_M \frac{1}{m}\left( \Delta _{b}u\right) ^{2}-%
\frac{4}{m}\left\vert \sum u_{\alpha ,\overline{\alpha }}\right\vert ^{2}-%
\sqrt{-1}\left( A_{\alpha \beta }u_{\overline{\alpha }}u_{\overline{\beta }%
}-A_{\overline{\alpha }\overline{\beta }}u_{\alpha }u_{\beta }\right),
\end{equation*}%
\begin{equation*}
\left( m-2\right) \int_M \sqrt{-1}\left( u_{\overline{\beta }}u_{\beta
,0}-u_{\beta }u_{\overline{\beta },0}\right) =\int_M 4\left\vert u_{\alpha ,%
\overline{\beta }}\right\vert ^{2}-\left( \Delta _{b}u\right) ^{2}+\sqrt{-1}%
m\left( A_{\alpha \sigma }u_{\overline{\alpha }}u_{\overline{\sigma }}-A_{%
\overline{\alpha }\overline{\sigma }}u_{\sigma }u_{\alpha }\right).
\end{equation*}
\end{lemma}

We can now state the main estimate on $\lambda_1$. For completion and future application in the next section, we also provide the
detail of the proof here.
\begin{theorem}\label{GL}
\bigskip Let $M$ be a closed pseudohermitian manifold of dimension $%
2m+1\geq 5$. Suppose for any $X\in T^{1,0}\left( M\right) $%
\begin{equation}
Ric\left( X,X\right) -\frac{m+1}{2}Tor\left( X,X\right) \geq \kappa
\left\vert X\right\vert ^{2},  \label{c-hyp}
\end{equation}
where $\kappa $ is a positive constant. Then the first eigenvalue of $%
-\Delta _{b}$ satisfies 
\begin{equation*}
\lambda _{1}\geq \frac{m}{m+1}\kappa.
\end{equation*}
\end{theorem}

\begin{remark}
In terms of our local unitary frame, the assumption (\ref{c-hyp}) means that
for any $X=f_{\alpha}T_{\alpha }$ 
\begin{equation*}
R_{\alpha \overline{\sigma }}f_{\sigma }f_{\overline{\alpha }}+\frac{m+1}{2}%
\sqrt{-1}\left[ A_{\alpha \sigma }f_{\overline{\alpha }}f_{\overline{\sigma }%
}-A_{\overline{\alpha }\overline{\sigma }}f_{\sigma }f_{\alpha }\right] \geq
\kappa \sum_{\alpha }\left\vert f_{\alpha }\right\vert ^{2}.
\end{equation*}
\end{remark}

\begin{proof}
Suppose $\,-\Delta _{b}u=\lambda _{1}u$. Applying the Bochner formula, we
have%
\begin{eqnarray*}
0 &=&\int \left\vert u_{\alpha ,\beta }\right\vert ^{2}+\left\vert u_{\alpha
,\overline{\beta }}\right\vert ^{2}-\lambda _{1}\left\vert \partial_b
u\right\vert ^{2} \\
&&+R_{\alpha \overline{\sigma }}u_{\sigma }u_{\overline{\alpha }}+\frac{m}{2}%
\sqrt{-1}\left[ A_{\alpha \sigma }u_{\overline{\alpha }}u_{\overline{\sigma }%
}-A_{\overline{\alpha }\overline{\sigma }}u_{\sigma }u_{\alpha }\right] \\
&&+\sqrt{-1}\left( u_{\overline{\beta }}u_{\beta ,0}-u_{\beta }u_{\overline{%
\beta },0}\right).
\end{eqnarray*}%
We write the last term as $c$ times the first identity plus $\left(
1-c\right) $ times the second identity of Lemma \ref{intfor},%
\begin{eqnarray*}
0 &=&\int \left\vert u_{\alpha ,\beta }\right\vert ^{2}+\left\vert u_{\alpha
,\overline{\beta }}\right\vert ^{2}-\lambda _{1}\left\vert \partial_b
u\right\vert ^{2} \\
&&+R_{\alpha \overline{\sigma }}u_{\sigma }u_{\overline{\alpha }}+\frac{m}{2}%
\sqrt{-1}\left[ A_{\alpha \sigma }u_{\overline{\alpha }}u_{\overline{\sigma }%
}-A_{\overline{\alpha }\overline{\sigma }}u_{\sigma }u_{\alpha }\right] \\
&&+\frac{2c}{m}\int \left( \left\vert u_{\alpha ,\overline{\beta }%
}\right\vert ^{2}-\left\vert u_{\alpha ,\beta }\right\vert ^{2}-R_{\alpha 
\overline{\sigma }}u_{\sigma }u_{\overline{\alpha }}\right) \\
&&+\int \frac{1-c}{m}\lambda _{1}^{2}u^{2}-\frac{4\left( 1-c\right) }{m}%
\left\vert \sum u_{\alpha ,\overline{\alpha }}\right\vert ^{2}-\left(
1-c\right) \sqrt{-1}\left( A_{\alpha \beta }u_{\overline{\alpha }}u_{%
\overline{\beta }}-A_{\overline{\alpha }\overline{\beta }}u_{\alpha
}u_{\beta }\right) \\
&=&\int \left( 1-\frac{2c}{m}\right) \left\vert u_{\alpha ,\beta
}\right\vert ^{2}+\left( 1+\frac{2c}{m}\right) \left\vert u_{\alpha ,%
\overline{\beta }}\right\vert ^{2}+\left( -1+\frac{2\left( 1-c\right) }{m}%
\right) \lambda _{1}\left\vert \partial_b u\right\vert ^{2} \\
&&+\left( 1-\frac{2c}{m}\right) R_{\alpha \overline{\sigma }}u_{\sigma }u_{%
\overline{\alpha }}+\left( \frac{m}{2}-1+c\right) \sqrt{-1}\left[ A_{\alpha
\sigma }u_{\overline{\alpha }}u_{\overline{\sigma }}-A_{\overline{\alpha }%
\overline{\sigma }}u_{\sigma }u_{\alpha }\right] \\
&&-\frac{4\left( 1-c\right) }{m}\left\vert \sum u_{\alpha ,\overline{\alpha }%
}\right\vert ^{2}.
\end{eqnarray*}
Since $\sum _{\alpha,\beta=1}^m \left\vert u_{\alpha ,\overline{\beta }}\right\vert ^{2}\geq
\left\vert \sum u_{\alpha ,\overline{\alpha }}\right\vert ^{2}/m$, we have 
\begin{eqnarray*}
0 &\geq &\int \left( 1-\frac{2c}{m}\right) \left\vert u_{\alpha ,\beta
}\right\vert ^{2}+\left( -1+\frac{2\left( 1-c\right) }{m}\right) \lambda
_{1}\left\vert \partial_b u\right\vert ^{2} \\
&&+\left( 1-\frac{2c}{m}\right) R_{\alpha \overline{\sigma }}u_{\sigma }u_{%
\overline{\alpha }}+\left( \frac{m}{2}-1+c\right) \sqrt{-1}\left[ A_{\alpha
\sigma }u_{\overline{\alpha }}u_{\overline{\sigma }}-A_{\overline{\alpha }%
\overline{\sigma }}u_{\sigma }u_{\alpha }\right] + \\
&&+\left[ \left( 1+\frac{2c}{m}\right) \frac{1}{m}-\frac{4\left( 1-c\right) 
}{m}\right] \left\vert \sum u_{\alpha ,\overline{\alpha }}\right\vert ^{2}.
\end{eqnarray*}%
We choose $c$ such that
\begin{equation*}
\left( 1+\frac{2c}{m}\right) \frac{1}{m}-\frac{4\left( 1-c\right) }{m}=0,
\end{equation*}%
i.e. $c=3m/\left( 4m+2\right) $. Then%
\begin{eqnarray*}
0 &\geq &\int \frac{2\left( m-1\right) }{2m+1}\left\vert u_{\alpha ,\beta
}\right\vert ^{2}-\frac{2\left( m^{2}-1\right) }{m\left( 2m+1\right) }%
\lambda _{1}\left\vert \partial_b u\right\vert ^{2} \\
&&+\frac{2\left( m-1\right) }{2m+1}\left\{ R_{\alpha \overline{\sigma }%
}u_{\sigma }u_{\overline{\alpha }}+\frac{m+1}{2}\sqrt{-1}\left[ A_{\alpha
\sigma }u_{\overline{\alpha }}u_{\overline{\sigma }}-A_{\overline{\alpha }%
\overline{\sigma }}u_{\sigma }u_{\alpha }\right] \right\} .
\end{eqnarray*}%
Therefore%
\begin{equation*}
\frac{\left( m-1\right) }{2m+1}\int 2\left( \kappa -\frac{m+1}{m}\lambda
_{1}\right) \left\vert \partial_b u\right\vert ^{2}+\left\vert u_{\alpha
,\beta }\right\vert ^{2}\leq 0.
\end{equation*}%
It follows that $\lambda _{1}\geq m\kappa /\left( m+1\right) $ when $m\geq 2$%
.
\end{proof}

\section{Equality case}

We now discuss the equality case. By scaling, we can assume $\kappa =\left(
m+1\right) /2$ and thus $\lambda =m/2$.

\begin{proposition}
\label{crhessian}If equality holds in Theorem \ref{GL}, we must have 
\begin{eqnarray}
u_{\alpha ,\beta } &=&0,  \label{u20} \\
u_{\alpha ,\overline{\beta }} &=&\left( -\frac{1}{4}u+\frac{\sqrt{-1}}{2}%
u_{0}\right) \delta _{\alpha \beta },  \label{u11} \\
u_{0,\alpha } &=&2A_{\alpha \sigma }u_{\overline{\sigma }}+\frac{\sqrt{-1}}{2%
}u_{\alpha },  \label{u0a} \\
u_{0,0} &=&-\frac{1}{4}u+\frac{4}{m}{\Im}A_{\alpha \sigma ,\overline{%
\alpha }}u_{\overline{\sigma }}.  \label{u00}
\end{eqnarray}%
Moreover, at any point where $\partial_b u\neq 0$%
\begin{equation}
\sqrt{-1}A_{\alpha \beta }=\frac{Q}{\left\vert \partial_b u\right\vert ^{4}}%
u_{\alpha }u_{\beta },  \label{torc}
\end{equation}%
where $Q=\sqrt{-1}A_{\alpha \sigma }u_{\overline{\alpha }}u_{\overline{%
\sigma }}$.
\end{proposition}

\begin{proof}
If equality holds, we must have $u_{\alpha ,\beta }=0$ and%
\begin{equation}
u_{\alpha ,\overline{\beta }}=f\delta _{\alpha \beta },  \label{u11e}
\end{equation}%
where $f$ is a complex-valued function. Taking conjugate of (\ref{u11e}) yields 
\begin{eqnarray*}
\overline{f}\delta _{\alpha \beta } &=&u_{\overline{\alpha },\beta } \\
&=&u_{\beta ,\overline{\alpha }}-\sqrt{-1}\delta _{\alpha \beta }u_{0} \\
&=&\left( f-\sqrt{-1}u_{0}\right) \delta _{\alpha \beta }.
\end{eqnarray*}%
Hence ${\Im}f=\frac{1}{2}u_{0}$. We also have 
\begin{eqnarray*}
\frac{m}{2}u &=&-\Delta _{b}u \\
&=&-u_{\alpha ,\overline{\alpha }}-u_{\overline{\alpha },\alpha } \\
&=&-m\left( f+\overline{f}\right) .
\end{eqnarray*}%
Thus ${\Re}f=-\frac{1}{4}u$. Therefore%
\begin{equation}
f=-\frac{1}{4}u+\frac{\sqrt{-1}}{2}u_{0},  \label{finu}
\end{equation}%
This proves (\ref{u11}).

Differentiating (\ref{u20}), we have 
\begin{eqnarray*}
0 &=&u_{\alpha ,\beta \gamma }-u_{\alpha ,\gamma \beta } \\
&=&\sqrt{-1}\left( A_{\alpha \gamma }u_{\beta }-A_{\alpha \beta }u_{\gamma
}\right) .
\end{eqnarray*}%
Therefore 
\begin{equation*}
A_{\alpha \gamma }u_{\beta }-A_{\alpha \beta }u_{\gamma }=0.
\end{equation*}%
From this we easily obtain (\ref{torc}).

Differentiating (\ref{u11e}), we have%
\begin{eqnarray*}
f_{\overline{\gamma }}\delta _{\alpha \beta } &=&u_{\alpha ,\overline{\beta }%
\overline{\gamma }} \\
&=&u_{\alpha ,\overline{\gamma }\overline{\beta }}+\sqrt{-1}\left( h_{\alpha 
\overline{\beta }}A_{\overline{\gamma }}^{\sigma }-h_{\alpha \overline{%
\gamma }}A_{\overline{\beta }}^{\sigma }\right) u_{\sigma } \\
&=&f_{\overline{\beta }}\delta _{\alpha \gamma }+\sqrt{-1}\left( \delta
_{\alpha \beta }A_{\overline{\gamma }}^{\sigma }-\delta _{\alpha \gamma }A_{%
\overline{\beta }}^{\sigma }\right) u_{\sigma },
\end{eqnarray*}%
i.e.$\left( f_{\overline{\gamma }}-\sqrt{-1}A_{\overline{\gamma }}^{\sigma
}u_{\sigma }\right) \delta _{\alpha \beta }=\left( f_{\overline{\beta }}-%
\sqrt{-1}A_{\overline{\beta }}^{\sigma }u_{\sigma }\right) \delta _{\alpha
\gamma }$. It follows $f_{\overline{\gamma }}-\sqrt{-1}A_{\overline{\gamma }%
}^{\sigma }u_{\sigma }=0$. Using (\ref{finu}), this yields%
\begin{equation}
\sqrt{-1}A_{\alpha \sigma }u_{\overline{\sigma }}=\frac{1}{2}u_{\alpha }+%
\sqrt{-1}u_{\alpha ,0}  \label{au}.
\end{equation}
This then implies (\ref{u0a}) by using the first identity of Proposition \ref%
{rule}.

To prove the last identity, taking trace of (\ref{u11}) we obtain 
\begin{equation*}
m\left( -\frac{1}{4}u+\frac{\sqrt{-1}}{2}u_{0}\right) =u_{\alpha ,\overline{%
\alpha }}.
\end{equation*}
Differentiating and using Proposition \ref{rule} yields%
\begin{eqnarray*}
m\left( -\frac{1}{4}u_{0}+\frac{\sqrt{-1}}{2}u_{0,0}\right) &=&u_{\alpha ,%
\overline{\alpha }0} \\
&=&u_{\alpha ,0\overline{\alpha }}-A_{\overline{\alpha }\overline{\sigma }%
,\alpha }u_{\sigma } \\
&=&\left( A_{\alpha \sigma }u_{\overline{\sigma }}+\frac{\sqrt{-1}}{2}%
u_{\alpha }\right) _{,\overline{\alpha }}-A_{\overline{\alpha }\overline{%
\sigma },\alpha }u_{\sigma } \\
&=&A_{\alpha \sigma ,\overline{\alpha }}u_{\overline{\sigma }}+\frac{\sqrt{-1%
}}{2}u_{\alpha ,\overline{\alpha }}-A_{\overline{\alpha }\overline{\sigma }%
,\alpha }u_{\sigma } \\
&=&2\sqrt{-1}{\Im}A_{\alpha \sigma ,\overline{\alpha }}u_{\overline{%
\sigma }}+m\frac{\sqrt{-1}}{2}\left( -\frac{1}{4}u+\frac{\sqrt{-1}}{2}%
u_{0}\right) \\
&=&-\frac{m}{4}u_{0}+\sqrt{-1}\left( 2{\Im}A_{\alpha \sigma ,\overline{%
\alpha }}u_{\overline{\sigma }}-\frac{m}{8}u\right) .
\end{eqnarray*}%
Taking the imaginary part yields (\ref{u00}).
\end{proof}

\begin{lemma}
We also have%
\begin{equation}
R_{\alpha \overline{\sigma }}u_{\sigma }+\left( m+1\right) \sqrt{-1}%
A_{\alpha \sigma }u_{\overline{\sigma }}=\frac{m+1}{2}u_{\alpha }.
\label{ra}
\end{equation}
\end{lemma}

\begin{proof}
This follows easily from the fact that equality is achieved by $X=u_{%
\overline{\sigma }}T_{\sigma }$ in (\ref{c-hyp}). We can also derive it in the following way.
Differentiating (\ref{u20}) and using (\ref{u11e}) yields%
\begin{eqnarray*}
0 &=&u_{\alpha ,\beta \overline{\gamma }} \\
&=&u_{\alpha ,\overline{\gamma }\beta }+\sqrt{-1}\delta _{\beta \gamma
}u_{\alpha ,0}-R_{\beta \overline{\gamma }\alpha \overline{\sigma }%
}u_{\sigma } \\
&=&\left( -\frac{1}{4}u_{\beta }+\frac{\sqrt{-1}}{2}u_{0,\beta }\right)
\delta _{\alpha \gamma }+\sqrt{-1}u_{\alpha ,0}\delta _{\beta \gamma
}-R_{\beta \overline{\gamma }\alpha \overline{\sigma }}u_{\sigma }.
\end{eqnarray*}%
Taking trace over $\beta $ and $\gamma $, we obtain%
\begin{eqnarray*}
0 &=&-\frac{1}{4}u_{\alpha }+\frac{\sqrt{-1}}{2}u_{0,\alpha }+m\sqrt{-1}%
u_{\alpha ,0}+R_{\alpha \overline{\sigma }}u_{\sigma } \\
&=&-\frac{1}{4}u_{\alpha }+\frac{\sqrt{-1}}{2}A_{\alpha \sigma }u_{\overline{%
\sigma }}+\left( m+\frac{1}{2}\right) \sqrt{-1}u_{\alpha ,0}+R_{\alpha 
\overline{\sigma }}u_{\sigma } \\
&=&-\frac{\left( m+1\right) }{2}u_{\alpha }+\left( m+1\right) \sqrt{-1}%
A_{\alpha \sigma }u_{\overline{\sigma }}+R_{\alpha \overline{\sigma }%
}u_{\sigma },
\end{eqnarray*}%
where in the last step we have used (\ref{au}) to replace $u_{\alpha ,0}$.
\end{proof}

\begin{lemma}
$Q$ is real and nonpositive.
\end{lemma}

\begin{remark}\label{aQ}
This lemma will not be needed in the proof of the rigidity. However, it
yields a quick proof if we assume the following extra condition%
\begin{equation*}
A_{\alpha \beta ,\overline{\alpha }\overline{\beta }}=0,
\end{equation*}%
i.e. the double divergence of the torsion is zero. Indeed, integrating by
parts and using (\ref{u20}) we obtain%
\begin{eqnarray*}
\int_{M}Q &=&-\sqrt{-1}\int_{M}A_{\alpha \beta ,\overline{\alpha }}u_{%
\overline{\beta }}u \\
&=&\frac{\sqrt{-1}}{2}\int_{M}A_{\alpha \beta ,\overline{\alpha }\overline{%
\beta }}u^{2} \\
&=&0.
\end{eqnarray*}%
As $Q$ is nonpositive, this implies that $Q=0$. Therefore $A=0$. See the
discussion on the torsion-free case below. 
\end{remark}

\begin{proof}
From (\ref{ra}) we \ have 
\begin{equation*}
R_{\alpha \overline{\sigma }}u_{\overline{\alpha }}u_{\sigma }+\left(
m+1\right) Q=\frac{m+1}{2}\left\vert \partial_b u\right\vert ^{2}.
\end{equation*}%
This shows that $Q$ is real. Taking conjugate, we also have $Q=-\sqrt{-1}A_{%
\overline{\alpha }\overline{\sigma }}u_{\alpha }u_{\sigma }$. In the
inequality (\ref{c-hyp}), taking $X=e^{it}u_{\overline{\alpha }}T_{\alpha }$ yields%
\begin{equation*}
R_{\alpha \overline{\sigma }}u_{\overline{\alpha }}u_{\sigma }+\left(
m+1\right) Q\cos 2t\geq \frac{m+1}{2}\left\vert \partial_b u\right\vert ^{2}.
\end{equation*}%
Therefore $Q\leq 0$.
\end{proof}

\bigskip

\bigskip Theorem \ref{Main} follows from

\begin{lemma}
\label{mainlem}The torsion $A=0$.
\end{lemma}

The proof of this statement will be presented in the next section.

Assuming this lemma, Theorem \ref{Main} then follows from Chang and Chiu \cite{CC1}. In
the following, we present a simpler and more direct argument. Since $A$
vanishes, we have 
\begin{eqnarray}
u_{0,\alpha } &=&u_{\alpha ,0}=\frac{\sqrt{-1}}{2}u_{\alpha },  \label{a0} \\
u_{\alpha ,\beta } &=&0,u_{0,0}=-\frac{1}{4}u \\
u_{\alpha ,\overline{\beta }} &=&\left( -\frac{1}{4}u+\frac{\sqrt{-1}}{2}%
u_{0}\right) \delta _{\alpha \beta }.
\end{eqnarray}

By Proposition \ref{hessian}, we obtain

\begin{proposition}
Let $D^{2}u$ be the Hessian of $u$ with respect to the Riemannian metric $%
g_{\theta }$. Then 
\begin{equation*}
D^{2}u=-\frac{1}{4}ug_{\theta }.
\end{equation*}
\end{proposition}

By Obata's theorem (Theorem \ref{OT}), $(M,g_{\theta})$ is isometric to the sphere $\mathbb{S}^{2m+1}$ with the metric $g_0=4g_c$, where $g_c$ is the
canonical metric. Without loss of generality, we can take $(M,g_{\theta})$ to be ($\mathbb{S}^{2m+1},g_0$). Then $\theta$ is a pseudohermitian structure on $\mathbb{S}^{2m+1}$  whose
adapted metric is $g_0$ and the associated Tanaka-Webster connection is torsion-free. It is a well known fact that the Reeb vector field $T$ is then a Killing vector field for  $g_0$ (this can be easily proved by the first formula in Remark \ref{one})  .
Therefore there exists a skew-symmetric matrix $A$ such that for all $X\in\mathbb{S}^{2m+1}, T(X)=AX$, here we use the obvious identification between 
$z=(z_1,\ldots,z_{m+1})\in \mathbb{C}^{m+1}$ and $X=(x_1,y_1,\ldots,x_{m+1},y_{m+1})\in \mathbb{R}^{2m+2}$. Changing coordinates by an orthogonal transformation we can assume that $A$ is of the following form
\[
A=\left[
\begin{array}
[c]{ccc}%
\begin{array}
[c]{cc}%
0 & a_{1}\\
a_{1} & 0
\end{array}
&  & \\
& \ddots & \\
&  &
\begin{array}
[c]{cc}%
0 & a_{m+1}\\
a_{m+1} & 0
\end{array}
\end{array}
\right]
\]
where $a_i\geq 0$. Therefore 
\[
T=\sum_{i}a_{i}\left(  y_{i}\frac{\partial}{\partial x_{i}}-x_{i}%
\frac{\partial}{\partial y_{i}}\right)
\]
Since $T$ is of unit length we must have 
\[4\sum_{i}a_{i}^2(x_i^2+y_i^2)=1\]
on $\mathbb{S}^{2m+1}$. Therefore all the $a_i$'s are equal to $1/2$. It follows that 
$$\theta=g_0(T,\cdot)=2\sqrt{-1}\overline{\partial}(|z|^2-1).$$
This finishes the proof of Theorem 4.

\section{Proof of Lemma \protect\ref{mainlem}}

Let $\psi =\left\vert A\right\vert $. In local unitary frame%
\begin{equation*}
\psi =\sqrt{\sum \left\vert A_{\alpha \beta }\right\vert ^{2}}.
\end{equation*}%
We note that $\psi $ is continuous and $\psi ^{2}$ is smooth. Let $K=\left\{
\partial_b u=0\right\} $. By (\ref{torc}), on $M\backslash K$\  $\psi $ is
smooth and 
\begin{equation*}
\psi =-\frac{Q}{\left\vert \partial_b u\right\vert ^{2}},
\end{equation*}%
or 
\begin{equation*}
A_{\alpha \beta }=\sqrt{-1}\psi \frac{u_{\alpha }u_{\beta }}{\left\vert
\partial_b u\right\vert ^{2}}.
\end{equation*}

\begin{lemma}
\label{Haus}The compact set $K$ is of Hausdorff dimension at most $n-2$ ($%
n=2m+1=\dim M)$. More precisely we have a countable union $K=\cup
_{i=1}^{\infty }E_{i}$, where each $E_{i}$ has finite $n-2$ dimensional
Hausdorff measure: $\mathcal{H}^{n-2}\left( E_{i}\right) <\infty $.
\end{lemma}

\begin{remark}
Here the Hausdorff dimension is defined using the distance function of the
Riemannian metric $g_{\theta }$.
\end{remark}

\begin{proof}
We have $K=K_{1}\sqcup K_{2}$, where%
\begin{equation*}
K_{1}=\left\{ p\in K:u\left( p\right) \neq 0\text{ or }u_{0}\left( p\right)
\neq 0\right\} ,K_{2}=\left\{ p\in M:u\left( p\right) =0\text{ and }du\left(
p\right) =0\right\} .
\end{equation*}%
We first prove that $K_{1}$ is of Hausdorff dimension $n-2$. Suppose $p\in
K_{1}$. In a local unitary frame $\left\{ T_{\alpha }\right\} $ we have by
Proposition \ref{crhessian} 
\begin{equation*}
u_{\alpha ,\overline{\beta }}=\left( -\frac{1}{4}u+\frac{\sqrt{-1}}{2}%
u_{0}\right) \delta _{\alpha \beta }.
\end{equation*}%
We write $T_{\alpha }=X_{\alpha }-\sqrt{-1}Y_{\alpha }$ in terms of the real
and imaginary parts. Then we have $2m$ real local vector fields $\left\{
Z_{i}\right\} $, where $Z_{i}=X_{i}$ for $i\leq m$ and $Z_{i}=Y_{i-m}$ for $%
i>m$. Along $K_{1}$ the above equation takes the following form%
\begin{eqnarray*}
X_{\beta }X_{\alpha }u+Y_{\beta }Y_{\alpha }u &=&-\frac{1}{4}u\delta
_{\alpha \beta }, \\
Y_{\beta }X_{\alpha }u-X_{\beta }Y_{\alpha }u &=&\frac{1}{2}u_{0}\delta
_{\alpha \beta }.
\end{eqnarray*}%
Since either $u\left( p\right) \neq $ or $u_{0}\left( p\right) \neq 0$, from
the above equation it is straightforward to check that there exists $i<j$
such that the local map $F:q\rightarrow \left( Z_{i}u\left( q\right)
,Z_{j}u\left( q\right) \right) $ from $M$ to $\mathbb{R}^{2}$ is of rank $2$
at $p$. By the implicit function theorem, $F^{-1}\left( 0\right) $ is a
codimension $2$ submanifold at $p$. As $K_{1}\subset F^{-1}\left( 0\right) $%
, we conclude that $K_{1}$ is of Hausdorff dimension at most $n-2$.

To handle $K_{2}$, we note that $u$ satisfies the following 2nd order
elliptic equation by Proposition \ref{crhessian}%
\begin{equation*}
\Delta u=-\left( \frac{m}{2}+\frac{1}{4}\right) u+\frac{4}{m}{\Im}%
A_{\alpha \sigma ,\overline{\alpha }}u_{\overline{\sigma }}.
\end{equation*}%
As $K_{2}$ is the singular nodal set of $u$, we have (see, e.g. \cite{HHL}) 
\begin{equation*}
\mathcal{H}^{n-2}\left( K_{2}\right) <\infty .
\end{equation*}
\end{proof}

\begin{lemma}
\label{re0}We have on $M\backslash K$ 
\begin{equation*}
{\Re}\sum \psi _{\alpha }u_{\overline{\alpha }}=0.
\end{equation*}
\end{lemma}

\begin{proof}
Let $v=Tu=u$\-$_{0}$. By the second formula of Lemma \ref{intfor} 
\begin{equation}
\begin{split}
\int_{M}\sqrt{-1}\left( v_{\overline{\beta }}v_{\beta ,0}-v_{\beta }v_{%
\overline{\beta },0}\right) =& \int_{M}\frac{1}{m}\left( \Delta _{b}v\right)
^{2}-\frac{4}{m}\left\vert \sum v_{\alpha ,\overline{\alpha }}\right\vert
^{2} \\
& -\sqrt{-1}\left( A_{\alpha \beta }v_{\overline{\alpha }}v_{\overline{\beta 
}}-A_{\overline{\alpha }\overline{\beta }}v_{\alpha }v_{\beta }\right).
\label{bfv}
\end{split}
\end{equation}
We will use Proposition \ref{crhessian} to simplify both sides.
On $M\backslash K$%
\begin{eqnarray*}
v_{\alpha } &=&u_{0,\alpha }=\sqrt{-1}\left( 2\psi +\frac{1}{2}\right)
u_{\alpha }, \\
v_{\alpha ,\overline{\beta }} &=&\sqrt{-1}2\psi _{\overline{\beta }%
}u_{\alpha }+\left( 2\psi +\frac{1}{2}\right) \left( -\frac{1}{2}u_{0}-\frac{%
\sqrt{-1}}{4}u\right) \delta _{\alpha \beta }.
\end{eqnarray*}%
Differentiating the first equation yields $v_{\alpha ,\beta }=2\sqrt{-1}\psi
_{\beta }u_{\alpha }$. As $v_{\alpha ,\beta }=v_{\beta ,\alpha }$, we have $%
\psi _{\beta }u_{\alpha }=\psi _{\alpha }u_{\beta }$. As a result, on $%
M\backslash K$ there are smooth real functions $a,b$ such that
\begin{equation*}
\psi _{\alpha }=\left( a+ib\right) u_{\alpha }.
\end{equation*}%
A simple calculation shows 
\begin{eqnarray*}
\sqrt{-1}\left( v_{\overline{\beta }}v_{\beta ,0}-v_{\beta }v_{\overline{%
\beta },0}\right) &=&\sqrt{-1}\left( 2\psi +\frac{1}{2}\right) ^{2}\left( u_{%
\overline{\beta }}u_{\beta ,0}-u_{\beta }u_{\overline{\beta },0}\right) \\
&=&-\left( 2\psi +\frac{1}{2}\right) ^{2}\left( 2\psi +1\right) \left\vert
\partial _{b}u\right\vert ^{2}.
\end{eqnarray*}%
The integrand of the right hand side can be simplified as following%
\begin{eqnarray*}
&&\frac{1}{m}\left( \Delta _{b}v\right) ^{2}-\frac{4}{m}\left\vert \sum
v_{\alpha ,\overline{\alpha }}\right\vert ^{2}-\sqrt{-1}\left( A_{\alpha
\beta }v_{\overline{\alpha }}v_{\overline{\beta }}-A_{\overline{\alpha }%
\overline{\beta }}v_{\alpha }v_{\beta }\right) \\
&=&-\frac{4}{m}\left\vert {\Im}\sum v_{\alpha ,\overline{\alpha }%
}\right\vert ^{2}-\sqrt{-1}\left( A_{\alpha \beta }v_{\overline{\alpha }}v_{%
\overline{\beta }}-A_{\overline{\alpha }\overline{\beta }}v_{\alpha
}v_{\beta }\right) \\
&=&-\frac{4}{m}\left( 2a\left\vert \partial _{b}u\right\vert ^{2}-\frac{m}{4}%
\left( 2\psi +\frac{1}{2}\right) u\right) ^{2}-\left( 2\psi +\frac{1}{2}%
\right) ^{2}2\psi \left\vert \partial _{b}u\right\vert ^{2} \\
&=&-\frac{m}{4}\left( 2\psi +\frac{1}{2}\right) ^{2}u^{2}+4a\left\vert
\partial _{b}u\right\vert ^{2}\left( 2\psi +\frac{1}{2}\right) u-\frac{16}{m}%
a^{2}\left\vert \partial _{b}u\right\vert ^{4} \\
&&-\left( 2\psi +\frac{1}{2}\right) ^{2}2\psi \left\vert \partial
_{b}u\right\vert ^{2}.
\end{eqnarray*}%
Integrating by parts (see the remark below) yields%
\begin{eqnarray*}
\int -\frac{m}{4}\left( 2\psi +\frac{1}{2}\right) ^{2}u^{2} &=&{\Re}\int
\left( 2\psi +\frac{1}{2}\right) ^{2}u_{\alpha ,\overline{\alpha }}u \\
&=&-\int \left( 2\psi +\frac{1}{2}\right) ^{2}\left\vert \partial
_{b}u\right\vert ^{2}-4{\Re}\int \left( 2\psi +\frac{1}{2}\right) u\psi
_{\overline{\alpha }}u_{\alpha } \\
&=&-\int \left( 2\psi +\frac{1}{2}\right) ^{2}\left\vert \partial
_{b}u\right\vert ^{2}-4\int \left( 2\psi +\frac{1}{2}\right) au\left\vert
\partial _{b}u\right\vert ^{2}.
\end{eqnarray*}%
Plugging these calculations in (\ref{bfv}), we obtain%
\begin{equation*}
\frac{16}{m}\int_{M}a^{2}\left\vert \partial _{b}u\right\vert ^{4}=0.
\end{equation*}%
Therefore ${\Re}\sum \psi _{\alpha }u_{\overline{\alpha }}=a\left\vert
\partial _{b}u\right\vert ^{2}=0$.
\end{proof}

\begin{remark} 
We can justify the integration by parts in the following way.
We note that the compact set $K$
has zero $2$-capacity by Lemma \ref{Haus} (cf. \cite{EG, HKM}). Therefore
there exists a sequence $\chi _{k}\in C_{c}^{\infty }\left( M\backslash
K\right) $ s.t. $\chi _{k}\rightarrow 1$ in $W^{1,2}\left( M\right) $. Then%
\begin{eqnarray*}
\begin{split}
-\int \left( 2\phi +\frac{1}{2}\right) ^{2}u_{\alpha ,\overline{\alpha }%
}u\left( \chi _{k}\right) ^{2} =& \int \left( 2\phi +\frac{1}{2}\right)
^{2}\left\vert \partial _{b}u\right\vert ^{2}\left( \chi _{k}\right) ^{2} \\
&+2\int \left( 2\phi +\frac{1}{2}\right) \phi _{\overline{\alpha }%
}u_{\alpha }u\left( \chi _{k}\right) ^{2}+E_{k}, \label{apx}
\end{split}
\end{eqnarray*}
where 
\begin{equation*}
E_{k}=2\int \left( 2\phi +\frac{1}{2}\right) ^{2}u_{\alpha }u\chi _{k}\left(
\chi _{k}\right) _{\overline{\alpha }}.
\end{equation*}%
It is easy to see that $\lim_{k\rightarrow \infty }E_{k}=0$. Therefore,
letting $k\rightarrow \infty $ yields%
\begin{equation*}
-\int \left( 2\phi +\frac{1}{2}\right) ^{2}u_{\alpha ,\overline{\alpha }%
}u=\int \left( 2\phi +\frac{1}{2}\right) ^{2}\left\vert \partial
_{b}u\right\vert ^{2}+4\int \left( 2\phi +\frac{1}{2}\right) \phi _{%
\overline{\alpha }}u_{\alpha }u.
\end{equation*}
\end{remark}

\bigskip

We now prove Lemma \ref{mainlem}. Suppose $\psi ^{2}$ is not identically
zero. Let $\varepsilon ^{2}$ be a regular value of $\psi ^{2}$ such that  $\left\{
\psi \geq \varepsilon \right\} $ is a nonempty domain with smooth boundary.
Define 
\begin{equation*}
F=\left\{ 
\begin{array}{cc}
\psi \left( \psi -\varepsilon \right) ^{2}, & \text{if }\psi \geq
\varepsilon ; \\ 
0 & \text{if }\psi <\varepsilon .%
\end{array}%
\right.
\end{equation*}%
Then $F\in W^{1,2}\left( M\right) $. Integrating by parts, we obtain%
\begin{equation}
\begin{split}
\int_{M}Fu^{2\left( k+1\right) } =&-\frac{4}{m}{\Re}\int_{M}F\left\vert
u\right\vert ^{2k+1}u_{\alpha ,\overline{\alpha }} \\
=&\frac{4\left( 2k+1\right) }{m}{\Re}\int_{M}Fu^{2k}\left\vert \partial
_{b}u\right\vert ^{2} \\
&+\frac{4}{m}\int_{\left\{ \psi \geq \varepsilon \right\} }\left( 3\psi
^{2}-4\varepsilon \psi +\varepsilon ^{2}\right) \left\vert u\right\vert
^{2k+1}{\Re}u_{\alpha }\psi _{\overline{\alpha }} \\
=&\frac{4\left( 2k+1\right) }{m}\int_{M}Fu^{2k}\left\vert \partial
_{b}u\right\vert ^{2},
\end{split}
\label{fu}
\end{equation}%
by Lemma \ref{re0}. Integrating by parts again, we have%
\begin{eqnarray*}
&&\int_{M}Fu^{2k}\left\vert \partial _{b}u\right\vert ^{2} \\
&=&\int_{M}u^{2k}\left( \psi -\varepsilon \right) _{+}^{2}\psi \left\vert
\partial _{b}u\right\vert ^{2} \\
&=&-{\Re}\sqrt{-1}\int_{M}u^{2k}\left( \psi -\varepsilon \right)
_{+}^{2}A_{\alpha \beta }u_{\overline{\alpha }}u_{\overline{\beta }} \\
&=&-{\Re}\frac{\sqrt{-1}}{2k+1}\int_{M}\left( \psi -\varepsilon \right)
_{+}^{2}A_{\alpha \beta }\left( u^{2k+1}\right) _{\overline{\alpha }}u_{%
\overline{\beta }} \\
&=&{\Re}\frac{\sqrt{-1}}{2k+1}\left( \int_{M}\left( \psi -\varepsilon
\right) _{+}^{2}A_{\alpha \beta ,\overline{\alpha }}u^{2k+1}u_{\overline{%
\beta }}+2\int_{\left\{ \psi \geq \varepsilon \right\} }\left( \psi
-\varepsilon \right) u^{2k+1}\psi _{\overline{\alpha }}A_{\alpha \beta }u_{%
\overline{\beta }}\right) \\
&=&{\Re}\frac{\sqrt{-1}}{2k+1}\int_{M}\left( \psi -\varepsilon \right)
_{+}^{2}A_{\alpha \beta ,\overline{\alpha }}u^{2k+1}u_{\overline{\beta }}-%
\frac{2}{2k+1}\int_{\left\{ \psi \geq \varepsilon \right\} }\left( \psi
-\varepsilon \right) \psi u^{2k+1}{\Re}\psi _{\overline{\alpha }%
}u_{\alpha } \\
&=&{\Re}\frac{\sqrt{-1}}{2k+1}\int_{M}\left( \psi -\varepsilon \right)
_{+}^{2}A_{\alpha \beta ,\overline{\alpha }}u^{2k+1}u_{\overline{\beta }},
\end{eqnarray*}%
by Lemma \ref{re0} again. Let $C$ be super norm of $\mathrm{div}A=A_{\alpha
\beta ,\overline{\alpha }}$. Then by the H\"older inequality%
\begin{eqnarray*}
\int_{M}Fu^{2k}\left\vert \partial _{b}u\right\vert ^{2} &\leq &\frac{C}{%
\varepsilon \left( 2k+1\right) }\int_{M}F\left\vert u\right\vert
^{2k+1}\left\vert \partial _{b}u\right\vert \\
&\leq &\frac{C}{\varepsilon \left( 2k+1\right) }\left( \int_{M}F\left\vert
u\right\vert ^{2\left( k+1\right) }\right) ^{1/2}\left( \int_{M}F\left\vert
u\right\vert ^{2k}\left\vert \partial _{b}u\right\vert ^{2}\right) ^{1/2}.
\end{eqnarray*}%
Hence%
\begin{eqnarray*}
\int_{M}Fu^{2k}\left\vert \partial _{b}u\right\vert ^{2} &\leq &\left[ \frac{%
C}{\varepsilon \left( 2k+1\right) }\right] ^{2}\int_{M}F\left\vert
u\right\vert ^{2\left( k+1\right) } \\
&\leq &\frac{4C^{2}}{\varepsilon ^{2}m\left( 2k+1\right) }%
\int_{M}Fu^{2k}\left\vert \partial _{b}u\right\vert ^{2},
\end{eqnarray*}%
where in the last step we used (\ref{fu}). Choosing $k$ such that  $\frac{4C^{2}}{%
\varepsilon ^{2}m\left( 2k+1\right) }\leq \frac{1}{2}$ yields%
\begin{equation*}
\int_{M}Fu^{2k}\left\vert \partial _{b}u\right\vert ^{2}=0.
\end{equation*}%
This is a contradiction. Therefore Lemma \ref{mainlem} is proved. 



Inspecting the proof of the rigidity, it is clear that we only need to have
a non-constant function $u$ satisfying (\ref{u20}) and (\ref{u11}) as all the other identities used in
the proof are derived from these two. In summary, we have proved the
following theorem. 

\begin{theorem}\label{rig}
Let $M$ be a closed pseudohermitian manifold of dimension $2m+1\geq 5$.
Suppose there exists  a non-constant function $u\in C^{\infty }\left(
M\right) $ satisfying 
\begin{eqnarray*}
u_{\alpha ,\beta } &=&0, \\
u_{\alpha ,\overline{\beta }} &=&\left( -\frac{1}{4}u+\frac{\sqrt{-1}}{2}%
u_{0}\right) \delta _{\alpha \beta }.
\end{eqnarray*}%
Then $M$ is equivalent to the sphere $(\mathbb{S}^{2m+1},2\sqrt{-1}\overline{\partial}(|z|^2-1))$. 
\end{theorem}
This is equivalent to Theorem \ref{crot} by scaling.
\section{Remarks for the case $2m+1=3$}

Generally speaking, $3$-dimensional CR manifolds are more subtle  to understand than higher dimensional ones. A famous example is the CR embedding problem.
In our situation, it is not clear if Theorem \ref{rig} is true in 3 dimensions. The reason is that (\ref{u0a}) does not follow from (\ref{u20}) and 
(\ref{u11}) in 3 dimensions (In deriving (\ref{au}) we need at least 2 indices). The arguments
in Section 5 do yield the following weaker rigidity theorem in dimension $3$  with (\ref{u0a}) as an extra condition.

\begin{theorem}\label{rig3}
Let $M$ be a $3$-dimensional closed pseudohermitian manifold. Suppose
there exists a non-constant function $u\in C^{\infty }\left( M\right) $
satisfying 
\begin{eqnarray*}
u_{1,1 } &=&0, \\
u_{1 ,\overline{1 }} &=&\left( -\frac{1}{4}u+\frac{\sqrt{-1}}{2}%
u_{0}\right), \\
u_{0,1} &=&2A_{11 }u_{\overline{1 }}+\frac{\sqrt{-1}}{2%
}u_{1 }.
\end{eqnarray*}%
Then $M$ is equivalent to the sphere $(\mathbb{S}^{3},2\sqrt{-1}\overline{\partial}(|z|^2-1))$.
\end{theorem}
 
In fact, the eigenvalue estimate (Theorems 3) is not known in the $3$-dimensional case without
any extra condition. Chang and Chiu in \cite{CC2} proved the estimate under the extra condition that the 
Panietz operator is nonnegative. They also proved that $M$ is CR equivalent to the sphere if equality holds and the torsion is zero.

Recall that the Panietz operator $P_0$  acting on functions on a pseudohermitian manifold $M$ of dimension $2m+1$ is defined by
\begin{equation*}
P_{0}u=4\Re\left( u_{\overline{\beta },\beta \alpha }+m\sqrt{-1}%
A_{\alpha \beta }u_{\overline{\beta }}\right) _{,\overline{\alpha} }.
\end{equation*}%
It is proved by Graham and Lee \cite{GL} that $P_0$ is always nonnegative if $M$ is closed and of dimension $\ge 5$ in the sense
$$\int_MuP_0u\ge 0$$
for any smooth function $u$. In 3 dimensions, $P_0$ is known to be nonnegative if the torsion is zero.

With our method, we can remove the torsion-free condition in the characterization of the equality case in Chang and Chiu's work. 
\begin{theorem} Let $(M^3,\theta)$ be a closed pseudohermitian manifold such that for any $X=c T_1$
\begin{equation}\label{cv3}
R_{1\1bar}|c|^2-\sqrt{-1}(A_{11}\overline{c}^2-A_{\1bar\1bar}c^2)\ge |c|^2.
\end{equation}
If  $P_0\ge 0$, then  $\lambda_1\ge \frac{1}{2}$ and the equality holds if and only if
$(M^3, \theta)$ is equivalent to $(\mathbb{S}^{3},2\sqrt{-1}\overline{\partial}(|z|^2-1))$.
\end{theorem}

\begin{remark} The first part of the theorem was proved by Chiu \cite{Ch} and the second part of the theorem was proved by Chang and Chiu \cite{CC2} under the extra condition
that $M$ is torsion free. 
\end{remark}

We sketch the proof here. By Lemma 2.2 in \cite{CC1}, one has
$$
\int_M u_0^2 =\int_M (\Delta_b u)^2-2\sqrt{-1}(A_{11}u_{\1bar}^2-A_{\1bar \1bar} u_1^2)]
-{1\over 2} \int_M u P_0 u.
$$
Using Proposition 3 and integrating by parts yields
\begin{eqnarray*}
\Re \Big[ \sqrt{-1} \int_M (u_{\1bar} u_{1,  0}-u_{1} u_{\1bar, 0}) \Big]
= - \int_M (u_{0})^2-\int_M \sqrt{-1}(A_{11} u_{\1bar}^2- A_{\1bar \1bar} u_{1}^2).\\
\end{eqnarray*}
Let $u$ be a non-zero first eigenfunction, 
$\Delta_b u=-\lambda_1 u.$
Then
$$
\int_M (\Delta_b u)^2=2 \lambda_1 \int_M |\d_b u|^2.
$$
By the Bochner formula (Theorem \ref{Boc}), Lemma \ref{intfor} and the above two identities, we have
\begin{eqnarray*}
0&=& \int_M |u_{1,1}|^2+|u_{1, \1bar}|^2 -\lambda_1|\d u|^2
+ R_{1 \1bar} u_{1} u_{\1bar} +{1\over 2} \sqrt{-1} [A_{1 1} u_{\1bar}^2 -A_{\1bar \1bar} u_1^2]\\
&&\quad +\Re \sqrt{-1} \int_M (u_{\1bar} u_{1, 0}-u_{1} u_{\1bar, 0}) \\
&=& \int_M |u_{1,1}|^2+{\lambda_1^2 u^2 \over 4}+{u_0^2 \over 4} -\lambda_1|\d u|^2
+ R_{1 \1bar} |u_1|^2 +{1\over 2} \sqrt{-1} [A_{1 1} u_{\1bar}^2 -A_{\1bar \1bar} u_{1}^2 ]\\
&&\quad - \int_M (u_{0})^2-\int_M \sqrt{-1}( A_{11} u_{\1bar}^2 -A_{\1bar \1bar} u_1^2) \\
&=& \int_M |u_{1,1}|^2 -{\lambda_1\over 2} |\d u|^2
+ R_{1 \1bar} |u_1|^2-{1\over 2} \sqrt{-1} [A_{11} u_{\1bar} ^2 -A_{\1bar \1bar} u_1^2] -{3\over 4} \int_M (u_{0})^2 \\
&=& \int_M |u_{1, 1}|^2 -{\lambda_1\over 2} |\d u|^2
+ R_{1 \1bar} |u_1|^2  -{1\over 2} \sqrt{-1} [A_{11} u_{\1bar}^2 -A_{\1bar \1bar} u_1^2] \\
&&-{3\over 4} \int_M [2 \lambda_1 |\d u|^2 -2 \sqrt{-1}  [A_{11} u_{\1bar}^2 -A_{\1bar \1bar} u_1^2] 
-{1\over 2} u P_0 u ] \\
\\
&=& \int_M |u_{1,1}|^2  +{3\over 2} u P_0 u+\int_M[-2 \lambda_1  |\d u|^2
+ R_{1 \1bar} |u_1|^2 + \sqrt{-1} [A_{11} u_{\1bar}^2 -A_{\1bar \1bar} u_1^2] \\
&\ge & \int_M[-2 \lambda_1  |\d u|^2
+ |\d u|^2] \\
\end{eqnarray*}
This implies that
$
\lambda_1\ge \frac{1}{2}
$.
If equality holds, we must have
\begin{equation}
u_{1,1}=0,\quad u_{1,\1bar}=-{u\over 4}+{\sqrt{-1}\over 2} u_0 \label{eq1}
\end{equation}
\begin{equation}
R_{1\1bar} |u_1|^2+\sqrt{-1} (A_{11} u_{\1bar}^2-A_{\1bar\1bar} u_1^2)=|u_1|^2.\label{eq2}
\end{equation}
Writing $A_{11}=e^{\sqrt{-1}\theta_1} |A_{11} |$, $u_1^2=e^{\sqrt{-1} \theta_2}|u_1|^2$ and $X=e^{\sqrt{-1}\theta} |X|$,  by (\ref{cv3}), we have
\begin{equation}
R_{1\1bar} |X|^2+\sqrt{-1} |A_{\1bar\1bar}| |X|^2(e^{\sqrt{-1} (\theta_1+\theta_2+\theta)}-e^{-\sqrt{-1} (\theta_1+\theta_2+\theta)})|\ge |X|^2,\label{eq3}
\end{equation}
for all $\theta\in [0,2\pi)$. Choosing $\theta=-\theta_1+{\pi\over 2}$, one has
\begin{equation}
R_{1\1bar} |X|^2-2|A_{11}| |X|^2 \ge |X|^2.\label{eq4}
\end{equation}
Therefore, by comparing  (\ref{eq2}) and (\ref{eq4})  with $X=u_1T_{\1bar}$, 
$$
\sqrt{-1}(A_{11} u_{\1bar}^2-A_{\1bar\1bar} u_1^2)=-2|A_{11}| |u_1|^2,\quad A_{11}=\sqrt{-1}|A_{11}| {u_1^2\over |u_1|^2}.
$$
Notice that
$$
R_{1\1bar} |u_1+tX|^2-2 |A_{11}||u_1+t X|^2-|u_1+tX|^2\ge 0,\quad \hbox{ on } \ M,
$$
and equality holds at $t=0$. Therefore we obtain by differentiating at $t=0$ 
$$ R_{1\1bar} u_1+2\sqrt{-1} A_{11} u_{\1bar} =u_1. $$
Using the 7th formula of Proposition \ref{rule}, we have
\begin{eqnarray*}
{-u_1\over 4}+{\sqrt{-1}\over 2} u_{0,1}&=& u_{1,\1bar 1}-u_{1,1 \1bar}\\
&=&-\sqrt{-1} u_{1,0} -R_{1\1bar} u_1\\
&=&-\sqrt{-1} u_{0,1} +\sqrt{-1} A_{11}u_{\1bar}-R_{1\1bar} u_1\\
&=&-\sqrt{-1} u_{0,1} +3\sqrt{-1} A_{11}u_{\1bar}-u_1
\end{eqnarray*}
Therefore,
$$
{1\over 2} u_{1} +\sqrt{-1} u_{0,1}=2\sqrt{-1} A_{11} u_{\1bar},\quad u_{0,1}={\sqrt{-1} \over 2} u_1+2 A_{11} u_{\1bar}. 
$$
Applying Theorem 9, the proof of Theorem 10 is complete.

\end{document}